\documentclass[a4paper,12pt]{amsart}

\newtheorem{proposition}{Proposition}
\newtheorem{theorem}{Theorem}
\newtheorem{lemma}{Lemma}

\newtheorem{remark}{Remark}

\numberwithin{equation}{section}
\usepackage{amssymb}

\begin{document}

\title[Positive solutions to RD eqn for PP model]{Positive solutions to the reaction diffusion equations for prey-predator models with dormancy of predators}
\author[Novrianti, Sawada, Tsuge]{Novrianti{$^{\ast, \sharp}$}, O.~Sawada{$^\ast$} and N.~Tsuge{$^\dagger$}}
\address{$^\ast$~Applied Physics Course, Faculty of Engineering, Gifu University, Yanagido 1-1, Gifu, 501-1193, Japan}
\address{$^\dagger$~Department of Mathematics Education, Faculty of Education, Gifu University, Yanagido 1-1, Gifu, 501-1193, Japan}
\address{$^\sharp$~Corresponding author}
\email{x3912006@edu.gifu-u.ac.jp}
\thanks{N. Tsuge's research is partially supported by Grant-in-Aid for Scientific Research (C) 17K05315, Japan.}

%%%%%%%%%%%%%%%%%%%%%%%%%%%%%%%%%%%%%%%%%%%%%%%%%%%%%%%%%%
%
% abstract
%
%%%%%%%%%%%%%%%%%%%%%%%%%%%%%%%%%%%%%%%%%%%%%%%%%%%%%%%%%%

\begin{abstract}
The time-global unique solvability on the reaction diffusion equations for prey-predator models with density-dependent inhibitor and dormancy on predators is established. The crucial step of the proof is to construct time-local non-negative classical solutions. To do so, new successive approximation and theories of time-evolution operators are used. Due to the maximum principle, the solutions are extended time-globally. Via analysis on the corresponding ordinary differential equations, invariant regions and asymptotic behaviors of solutions are also investigated.
\end{abstract}

\maketitle

%%%%%%%%%%%%%%%%%%%%%%%%%%%%%%%%%%%%%%%%%%%%%%%%%%%%%%%%%%
%
% introduction
%
%%%%%%%%%%%%%%%%%%%%%%%%%%%%%%%%%%%%%%%%%%%%%%%%%%%%%%%%%%

\section{introduction and main results}

\noindent We deal with the reaction diffusion equations in ${\mathbb R}^n$ for $n \in {\mathbb N}$:
\[
  ({\rm{LV}}) \,\,\, \left\{
  \begin{array}{ll}
    \displaystyle \partial_t u = \delta \Delta u + r \left( 1 - \frac{u}{k} \right) u - \gamma \frac{u v}{u+h}, & \\[9pt]
    \displaystyle \partial_t v = d \Delta v + \mu (u) \frac{u v}{u+h} + \alpha w - \theta v - \iota v - \beta v^2, & \\[9pt]
    \displaystyle \partial_t w = \nu (u) \frac{u v}{u+h}  + \theta v - \alpha w - \tilde \iota w. &
  \end{array}
  \right.
\]
This is a prey-predator model with dormancy of predators, see \cite{K15, KNO09}. Here, three variables $u = u (x, t)$, $v = v (x, t)$ and $w = w (x, t)$ stand for the unknown scalar functions at $x \in {\mathbb R}^n$ and $t > 0$ who denote densities of prey, active predator and dormant predator, respectively. We denote the diffusion coefficient of prey by $\delta$, the diffusion coefficient of active predator by $d$, the growth rate of prey by $r$, the capacity of prey by $k$, the mortality rate of prey by $\gamma$, the constant of foraging efficiency and handling time by $h$, the rate of awakening by $\alpha$, the rate of sleeping by $\theta$, the mortality rate of active predator by $\iota$, the mortality rate of dormant predator by $\tilde \iota$, the mortality rate by combats of active predators by $\beta$. Also, $\mu (u)$ and $\nu (u)$ are smooth positive functions of $u$ denoting growth rates of active and dormant predators, respectively. In \cite{K15}, $\mu$ is given as a sigmoid function $\mu (u) := \gamma (1+\tanh (\xi (u-\eta)))/2 \in (0, \gamma)$ with some constants $\xi$ and $\eta$; $\nu (u) := \gamma - \mu (u)$. We have used the notation of differentiation; $\partial_t := \partial / \partial t$ and $\Delta := \sum_{i=1}^n \partial_i^2$, where $\partial_i := \partial / \partial x_i$ for $i = 1, \ldots, n$.

By change of variables and constants, we can replace $\delta = 1$, $k = 1$, $r = 1$ and $\beta = 1$. For the simplicity of notation, we put $m := \theta + \iota$, $\rho := \alpha + \tilde \iota$, in addition, assume that $\mu$ and $\nu$ are positive constants independent of $u$. So, we consider the following initial value problem:
\[
  ({\rm{P}}) \,\,\, \left\{
  \begin{array}{ll}
    \partial_t u = \Delta u + (1 - u) u - \gamma u v / (u+h)  & {\text{in}} \,\, {\mathbb R}^n \!\times\! (0, \infty), \\[4pt]
    \partial_t v = d \Delta v + \mu u v / (u+h) + \alpha w - (m + v) v & {\text{in}} \,\, {\mathbb R}^n \!\times\! (0, \infty), \\[4pt]
    \partial_t w = \nu u v / (u+h) + \theta v - \rho w & {\text{in}} \,\, {\mathbb R}^n \!\times\! (0, \infty), \\[4pt]
    \big( u, v, w \big) \big|_{t=0} = \big( u_0, v_0, w_0 \big) & {\text{in}} \,\, {\mathbb R}^n.
  \end{array}
  \right.
\]

In \cite{KNO09}, the bifurcation between stability and instability of stationary solutions to (LV) was concerned with some specific parameters, associated with numerical investigation. In \cite{K15}, a numerical study of Turing instability on (LV) was done. Besides, in this paper, we focus into the mathematical theory for the existence of time-global non-negative unique classical solutions to (P), and the invariant region which includes the trivial solution $\big( 0, 0, 0 \big)$. We now state the main results.

%%%%%%%%%%%%%%%%%%%%%%%%%%%%%%%%%%%%%%%%%%%%%%%%%%%%%%%%%%
%
% Theorem 1
%
%%%%%%%%%%%%%%%%%%%%%%%%%%%%%%%%%%%%%%%%%%%%%%%%%%%%%%%%%%

\begin{theorem}\label{th}
Let $n \in {\mathbb N}$, $d$, $h > 0$, and let $m$, $\theta$, $\rho$, $\alpha$, $\gamma$, $\mu$, $\nu \geq 0$.
If $u_0, v_0 \in BUC ({\mathbb R}^n)$ and $w_0 \in BUC^1 ({\mathbb R}^n)$ are non-negative, then there exists a triplet $\big( u, v, w \big)$ of time-global unique classical solutions to {\rm{(P)}}.
\end{theorem}

%%%%%%%%%%%%%%%%%%%%%%%%%%%%%%%%%%%%%%%%%%%%%%%%%%%%%%%%%%
%
% Remark 1
%
%%%%%%%%%%%%%%%%%%%%%%%%%%%%%%%%%%%%%%%%%%%%%%%%%%%%%%%%%%

\begin{remark}\label{r1}{\rm
(i)~One can find at most five stationary constant states (solutions independent of $x$ and $t$), including the trivial solution and $\big( 1, 0, 0 \big)$. The trivial solution is always instable. Besides, the stabilities of non-trivial constant states depend on parameters; see Remark~\ref{r4}.

\noindent (ii)~Even if $\mu$ and $\nu$ are positive smooth functions of $u$, the same time-global solvability can be proved.

\noindent (iii)~When the initial data belong to $L^\infty$, one may get the same assertion, although there is a lack of continuity of solutions in $t$ at $t = 0$.
}\end{remark}

%%%%%%%%%%%%%%%%%%%%%%%%%%%%%%%%%%%%%%%%%%%%%%%%%%%%%%%%%%
%
% integral equations
%
%%%%%%%%%%%%%%%%%%%%%%%%%%%%%%%%%%%%%%%%%%%%%%%%%%%%%%%%%%

We will explain the strategy of the proof of Theorem~\ref{th}, briefly. Using the heat semigroups, (P) is written as the forms of integral equations:
\begin{align}
    u(t) & = e^{t \Delta} u_0 + \int_0^t e^{(t-s) \Delta} \left[ \left( 1 - u \right) u - \gamma \frac{u v}{u + h} \right] \! (s) \, ds, \label{int-u} \\
    v(t) & = e^{d t \Delta} v_0 + \int_0^t e^{d (t-s) \Delta} \left[ \mu\frac{u v}{u + h} + \alpha w - (m + v) v \right] \! (s) \, ds, \label{int-v} \\
    w(t) & = e^{-\rho t} w_0 + \int_0^t e^{-\rho (t-s)} \left[ \nu \frac{u v}{u + h} + \theta v \right] \! (s) \, ds. \label{int-w}
\end{align}
Although these forms are benefit to show the uniqueness and regularity of solutions, the non-negativity of solutions are not ensured, as long as one uses the standard successive approximation. Thus, we have to look for the other integral forms for proving the existence of non-negative solutions. To do so, we shall construct a triplet of the solutions $\big( u, v, w \big)$ as the limits of the following successive approximation:
\begin{align}
    u_{\ell + 1} (t) & = U_\ell (t, 0) u_0 + \int_0^t U_\ell (t, s) \left[ u_\ell \right] (s) \, ds, \label{iu2} \\
    v_{\ell + 1} (t) & = V_\ell (t, 0) v_0 + \int_0^t V_\ell (t, s) \left[ \mu \frac{u_\ell v_\ell}{u_\ell + h} + \alpha w_\ell \right] \! (s) \, ds, \label{iv2} \\
    w_{\ell + 1} (t) & = e^{-\rho t} w_0 + \int_0^t e^{-\rho (t-s)} \left[ \nu \frac{u_\ell v_\ell}{u_\ell + h} + \theta v_\ell \right] \! (s) \, ds \label{iw2}
\end{align}
for $\ell \in {\mathbb N}$, starting at
\begin{equation}\label{1st}
  u_1 (t) := e^{t \Delta} u_0, \quad v_1 (t) := e^{t (d  \Delta - m)} v_0 \quad {\text{and}} \quad w_1 (t) := e^{- \rho t} w_0.
\end{equation}
Here, $\big\{ U_\ell (t, s) \big\}_{t \geq s \geq 0}$ and $\big\{ V_\ell (t, s) \big\}_{t \geq s \geq 0}$ are time-evolution operators associated with $A_\ell := \Delta - u_\ell - \gamma v_\ell / (u_\ell + h)$ and $B_\ell := d \Delta - m  - v_\ell$ for regarding $u_\ell$, $v_\ell$ and $w_\ell$ as given non-negative functions, respectively. These approximation enable us to get non-negativities of $\big( u_\ell, v_\ell, w_\ell \big)$ for each $\ell \in {\mathbb N}$, as well as its limit $\big( u, v, w \big)$. The definition and estimates of time-evolution operators are given in Section~2.

On the other hand, it is rather standard to extend the obtained solutions time-globally, deriving a priori estimates of solutions. The key idea is to apply the maximum principle to the classical solutions. One may also investigate asymptotic behaviors of solutions, more precisely. Via analysis of solutions to the system of corresponding ordinary differential equations, the invariant sets are prescribed as follows.

%%%%%%%%%%%%%%%%%%%%%%%%%%%%%%%%%%%%%%%%%%%%%%%%%%%%%%%%%%
%
% Theorem 2
%
%%%%%%%%%%%%%%%%%%%%%%%%%%%%%%%%%%%%%%%%%%%%%%%%%%%%%%%%%%

\begin{theorem}\label{th2}
{\rm{(i)}}~Let $\overline v := \mu/(1+h) + \alpha (\nu + \theta + \theta h)/(\rho + \rho h) - m \leq 0$. If $u_0 \not\equiv 0$, then $\big( u, v, w \big) \to \big( 1, 0, 0 \big)$ as $t \to \infty$. Besides, if $u_0 \equiv 0$, then the solution tends to the trivial solution as $t \to \infty$.

\noindent {\rm{(ii)}}~Let $\overline v > 0$ and $\overline w := (\nu + \theta + \theta h) \overline v / (\rho + \rho h)$. For $\varepsilon > 0$, there exists a $T_\varepsilon \geq 0$ such that $\big( u, v, w \big) \in [ 0, 1 + \varepsilon ) \times [ 0, \overline v + \varepsilon ) \times [ 0, \overline w + \varepsilon )$ for $x \in {\mathbb R}^n$ and $t \geq T_\varepsilon$. Moreover, if $\big( u_0, v_0, w_0 \big) \in R := [ 0, 1 ] \times [ 0, \overline v ] \times [ 0, \overline w ]$, then $\big( u, v, w \big) \in R$ for $t > 0$.

\noindent {\rm{(iii)}}~Assume $\overline v > 0$ and
\begin{align*}
  \underline u & := (1-h)/2 + \sqrt{(1+h)^2 - 4 \gamma \overline v}/2 > 0, \\
  \underline v & := \mu \underline u \, / (\underline u + h) + \alpha \nu \underline u \, / (\rho \underline u + \rho h) + \alpha \theta / \rho - m > 0, \\
  \underline w & := \nu \underline u \, \underline v \, / (\rho \underline u + \rho h) + \theta \underline v \, / \rho > 0.
\end{align*}
Let $u$, $v$, $w \geq \underline c$ for $x \in {\mathbb R}^n$ at $t = \underline t \geq 0$ with some $\underline c > 0$. For $\varepsilon > 0$, there exists a $T_\varepsilon' \geq \underline t$ such that $\big( u, v, w \big) \in ( \underline u - \varepsilon, 1 + \varepsilon ) \times ( \underline v - \varepsilon, \overline v + \varepsilon ) \times ( \underline w - \varepsilon, \overline w + \varepsilon )$ for $x \in {\mathbb R}^n$ and $t \geq T_\varepsilon'$. Moreover, if $\big( u_0, v_0, w_0 \big) \in R_\natural := [ \underline u, 1 ] \times [ \underline v, \overline v ] \times [ \underline w, \overline w ]$, then $\big( u, v, w \big) \in R_\natural$ for $t > 0$.
\end{theorem}

The sets $R$ and $R_\natural$ are invariant regions. The reader may find another (narrower) invariant regions for each individual parameters. Theorem~\ref{th2} implies that an absorving set always exists in $R$ or $R_\natural$.

The authors believe that one can also obtain the similar results in several domains with suitable boundary conditions.

This paper is organized as follows. In Section~2, we shall define function spaces, and recall some properties of the heat semigroup and time-evolution operators. Section~3 will be devoted to show the time-local existence of non-negative unique classical solutions with non-negative initial data. We shall discuss the time-global solvability in Section~4, deriving a priori estimates of solutions and their derivatives. In Section~5, some invariant regions and asymptotic behaviors of solutions to (P) will be argued.

Throughout this paper, we denote positive constants by $C$ the value of which may differ from one occasion to another.

%%%%%%%%%%%%%%%%%%%%%%%%%%%%%%%%%%%%%%%%%%%%%%%%%%%%%%%%%%
%
% Acknowledgment
%
%%%%%%%%%%%%%%%%%%%%%%%%%%%%%%%%%%%%%%%%%%%%%%%%%%%%%%%%%%

%\vspace{2mm}
%\noindent
%{\bf Acknowledgment.} ...

%%%%%%%%%%%%%%%%%%%%%%%%%%%%%%%%%%%%%%%%%%%%%%%%%%%%%%%%%%
%
% Section 2. semigroups and time evolution operators
%
%%%%%%%%%%%%%%%%%%%%%%%%%%%%%%%%%%%%%%%%%%%%%%%%%%%%%%%%%%

\section{semigroups and time-evolution operators}

\noindent In this section, we recall definition of function spaces and properties of the heat semigroup as well as time-evolution operators.

Let $n \in {\mathbb N}$, $1 \leq p < \infty$, and let $L^p := L^p ({\mathbb R}^n)$ be the space of all $p$-th integrable functions in ${\mathbb R}^n$ with the norm $\displaystyle \| f \|_p := \left( \int_{{\mathbb R}^n} |f(x)|^p dx \right)^{1/p}$. We often omit the notation of domain $({\mathbb R}^n)$, if no confusion occurs likely. We do not distinguish scalar valued functions and vector, as well as function spaces. Let $L^\infty$ be the space of all bounded functions with the norm $\| f \| := \| f \|_\infty := {\rm{ess}}.\sup_{x \in {\mathbb R}^n} | f(x) |$; $BUC$ as the space of all bounded uniformly continuous functions. For $k \in {\mathbb N}$, let $W^{k, \infty}$ be a set of all bounded functions whose $k$-th derivatives are also bounded.

In the whole space ${\mathbb R}^n$, for $\vartheta_0 \in L^\infty ({\mathbb R}^n)$, the heat equation
\[
  ({\rm{H}}) \,\, \left\{
  \begin{array}{ll}
    \displaystyle \partial_t \vartheta = \Delta \vartheta \, & {\text{in}} \quad {\mathbb R}^n \!\times\! (0, \infty), \\
    \vartheta|_{t=0} = \vartheta_0 \, & {\text{in}} \quad {\mathbb R}^n
  \end{array}
  \right.
\]
admits a time-global unique smooth solution
\begin{align*}
  \vartheta & := \vartheta (t) := \vartheta (x, t) := (e^{t \Delta} \vartheta_0) (x) := (H_t \ast \vartheta_0) (x) \\
  & := \int_{{\mathbb R}^n} (4 \pi t)^{-n/2} \exp(-|x-y|^2/4t) \vartheta_0 (y) dy
\end{align*}
in $C_w ((0, \infty); L^\infty ({\mathbb R}^n))$, that is, $\vartheta \in C([\tau, \infty); L^\infty ({\mathbb R}^n))$ for any $\tau > 0$. Here, $H_t := H_t (x) := (4 \pi t)^{-n/2} \exp(-|x|^2/4t)$ is the heat kernel. Since $\| H_t \|_1 = 1$ for $t > 0$, by Young's inequality we have $\| \vartheta (t) \|_\infty \leq \| \vartheta_0 \|_\infty$ for $t > 0$. In particular, if $\vartheta_0 (x) \geq 0$ for all $x \in {\mathbb R}^n$, then $\vartheta (x, t) \geq 0$ holds true for $x \in {\mathbb R}^n$ and $t > 0$; so-called the maximum principle. Furthermore, if additionally $\vartheta_0 \in BUC ({\mathbb R}^n)$ and $\vartheta_0 \not\equiv 0$, then $\vartheta (x, t) > 0$ for $x \in {\mathbb R}^n$ and $t > 0$; so-called the strong maximum principle. For $\vartheta_0 \in L^\infty ({\mathbb R}^n)$, there is a lack of the continuity of solutions to (H) in time at $t = 0$, in general. Note that $e^{t \Delta} \vartheta_0 \to \vartheta_0$ in $L^\infty$ as $t \to 0$, if and only if $\vartheta_0 \in BUC ({\mathbb R}^n)$. The reader may find its proof in e.g. \cite{GIM99}. Indeed, if $\vartheta_0 \in BUC ({\mathbb R}^n)$, then $\vartheta \in C([0, \infty); BUC ({\mathbb R}^n))$.

One can easily see that for $j \in {\mathbb N}$, there exists a positive constant $C$ such that $\| \partial_i^j e^{t \Delta} \vartheta_0 \|_\infty \leq C t^{-j/2} \| \vartheta_0 \|_\infty$ for $t > 0$ and $1 \leq i \leq n$. So, $\vartheta (t) \in C^j ({\mathbb R}^n)$ for $j \in {\mathbb N}$ and $t >0$, which implies that $\vartheta (t) \in C^\infty ({\mathbb R}^n)$ for $t >0$. Moreover, $\vartheta \in C^\infty ({\mathbb R}^n \times (0, \infty))$ by using (H).

In what follows, we recall some properties and estimates for time-evolution operators. Consider the following autonomous problem:
\[
  ({\rm{P_A}}) \,\, \left\{
  \begin{array}{ll}
    \partial_t \varphi = d \Delta \varphi - \psi (x, t) \varphi \,
    & {\text{in}} \quad {\mathbb R}^n \!\times\! (0, \infty), \\
    \varphi|_{t=0} = \varphi_0 \, & {\text{in}} \quad {\mathbb R}^n.
  \end{array}
  \right.
\]
Here, $\psi (x, t)$ is a given bounded function. We establish the time-local solvability of $({\rm{P_A}})$ with upper bounds of $\varphi (t)$.

%%%%%%%%%%%%%%%%%%%%%%%%%%%%%%%%%%%%%%%%%%%%%%%%%%%%%%%%%%
%
% Lemma 1
%
%%%%%%%%%%%%%%%%%%%%%%%%%%%%%%%%%%%%%%%%%%%%%%%%%%%%%%%%%%

\begin{lemma}[\cite{KNST19}]\label{lem}
Let $n \in {\mathbb N}$, $d, T > 0$ and $\psi \in L^\infty ([0, T]; W^{1, \infty} ({\mathbb R}^n))$. If $\varphi_0 \in BUC ({\mathbb R}^n)$, then there exist a $T_\ast \in (0, T]$ and a time-local unique classical solution to $({\rm{P_A}})$, having $\displaystyle \| \varphi (t) \|_\infty \leq \frac{4}{3} \| \varphi_0 \|_\infty$ for $t \in [0, T_\ast]$. Moreover, if $\varphi_0 \geq 0$, then $\varphi \geq 0$.
\end{lemma}

\begin{proof}
Although the proof is written in \cite{KNST19}, we give it in here. The idea is to use the standard iteration. Let $\varphi_1 (t) := e^{d t \Delta} \varphi_0$, and let
\[
  \varphi_{\ell + 1} (t) := e^{d t \Delta} \varphi_0 - \int_0^t e^{d (t - s) \Delta} \left[ \psi \varphi_\ell \right] (s) \, ds
\]
for each $\ell \in {\mathbb N}$, successively. It is easy to see that $\displaystyle \| \varphi_\ell (t) \|_\infty \leq \frac{4}{3} \| \varphi_0 \|_\infty$ for $t \in [0, T_\ast]$ with some $T_\ast > 0$ (independent of $\ell$) and $\ell \in {\mathbb N}$. One may easily show $\big\{ \varphi_\ell \big\}_{\ell = 1}^\infty$ is a Cauchy sequence in $C([0, T_\ast]; BUC ({\mathbb R}^n))$. So, its limit $\varphi := \lim_{\ell \to \infty} \varphi_\ell$ exists, and satisfies $({\rm{P_A}})$, having the estimate $\displaystyle \| \varphi (t) \|_\infty \leq \frac{4}{3} \| \varphi_0 \|_\infty$ for $t \in [0, T_\ast]$. It is rather straightforward to obtain the uniqueness and regularity of $\varphi$. Moreover, the non-negativity of $\varphi$ easily follows from the maximum principle.
\end{proof}

Note that if $\| \varphi_0 \| \leq L$ and $\sup_{0 \leq t \leq T} \| \psi (t) \| \leq L$ with some $L > 0$, then we may derive the estimate $T_\ast \geq C/L$ with $C > 0$.

The solution to $({\rm{P_A}})$ can be rewritten as $\varphi (t) = U (t, 0) \varphi_0$, using time-evolution operators $\big\{ U (t, s) \big\}_{t \geq s \geq 0}$ associated with $A := A(x, t) := d \Delta - \psi (x, t)$; see in e.g. the book of Tanabe \cite{Tanabe79}. The boundedness of solutions $\varphi$ implies that $\| U (t, 0) \|_{L^\infty \to L^\infty} \leq 4/3$ for $t \in [0, T_\ast]$, and then $\| U (t, s) \|_{L^\infty \to L^\infty} \leq 4/3$ for $0 \leq s \leq t \leq T_\ast$. Here, we have used the notation of an operator-norm $\| \mathcal O \|_{X \to Y} := \sup_{x \in X} \| {\mathcal O} x \|_Y / \| x \|_X$.

%%%%%%%%%%%%%%%%%%%%%%%%%%%%%%%%%%%%%%%%%%%%%%%%%%%%%%%%%%
%
% Section 3. time-local solvability
%
%%%%%%%%%%%%%%%%%%%%%%%%%%%%%%%%%%%%%%%%%%%%%%%%%%%%%%%%%%

\section{time-local solvability}

\noindent We shall give a proof of the time-local solvability on (P) in this section. Let us denote by $\| \cdot \| := \| \cdot \|_\infty$.

%%%%%%%%%%%%%%%%%%%%%%%%%%%%%%%%%%%%%%%%%%%%%%%%%%%%%%%%%%
%
% Propostion 1. time-local well-posedness
%
%%%%%%%%%%%%%%%%%%%%%%%%%%%%%%%%%%%%%%%%%%%%%%%%%%%%%%%%%%

\begin{proposition}\label{tlwp}
Assume that $n \in {\mathbb N}$, $d > 0$, and that other parameters are non-negative. Let $u_0$, $v_0 \in BUC ({\mathbb R}^n)$ and $w_0 \in BUC^1 ({\mathbb R}^n)$. Put $M := \max \{ \| u_0 \|, \| v_0 \|, \| w_0 \|, \| \partial_i w_0 \| \}$. If $u_0$, $v_0$ and $w_0$ are non-negative, then there exist a positive time $T_0$ and a triplet $\big( u, v, w \big)$ of time-local unique classical solutions to {\rm{(P)}} in $C([0, T_0]; BUC ({\mathbb R}^n))$, having $0 \leq u (x, t), v (x, t), w (x, t) \leq 2 M$ for $x \in {\mathbb R}^n$ and $t \in [0, T_0]$. Furthermore, $T_0 \geq C/(M^4+1)$ with some $C > 0$ independent of $M$.
\end{proposition}

\begin{proof}
For the sake of simplicity, we assume that all parameter is positive. Making the approximation sequences, we begin with (\ref{1st}). For $\ell \in {\mathbb N}$, we successively define $u_{\ell + 1}$, $v_{\ell + 1}$ and $w_{\ell + 1}$ by $(\ref{iu2})-(\ref{iw2})$. So, $\big( u_{\ell + 1}, v_{\ell + 1}, w_{\ell + 1} \big)$ formally satisfies
\[
  \left\{
  \begin{array}{ll}
  \partial_t u_{\ell + 1} = \Delta u_{\ell + 1} + \left( 1 - u_{\ell + 1} \right) u_\ell - \gamma u_{\ell + 1} v_\ell / (u_\ell+h), & \\[4pt]
  \partial_t v_{\ell + 1} = d \Delta v_{\ell + 1} + \mu u_\ell v_\ell / (u_\ell + h) + \alpha w_\ell - (m + v_\ell) v_{\ell + 1}, & \\[4pt]
  \partial_t w_{\ell + 1} = \nu u_\ell v_\ell / (u_\ell + h) + \theta v_\ell - \rho w_{\ell + 1}, & \\[4pt]
  \big( u_{\ell + 1}, v_{\ell + 1}, w_{\ell + 1} \big) \big|_{t=0} = \big( u_0, v_0, w_0 \big) &
  \end{array}
  \right.
\]
for $x \in {\mathbb R}^n$ and $t >0$ with non-negative functions $u_0$, $v_0$, $w_0$, $u_\ell$, $v_\ell$, $w_\ell$.

In what follows, we estimate $u_\ell$, $v_\ell$, $w_\ell$, $\partial_i u_\ell$, $\partial_i v_\ell$ and $\partial_i w_\ell$. Put
\[
  \begin{array}{ll}
  K_{1, \ell} := \sup_{0 \leq t \leq T} \| u_\ell (t) \|, & K_{2, \ell} := \sup_{0 \leq t \leq T} \| v_\ell (t) \|, \\[5pt]
  K_{3, \ell} := \sup_{0 \leq t \leq T} \| w_\ell (t) \|, & K_{4, \ell} := \sup_{0 \leq t \leq T} t^{1/2} \| \partial_i u_\ell (t) \|, \\[5pt]
  K_{5, \ell} := \sup_{0 \leq t \leq T} (dt)^{1/2} \| \partial_i v_\ell (t) \|, & K_{6, \ell} := \sup_{0 \leq t \leq T} \| \partial_i w_\ell (t) \|
  \end{array}
\]
for $T > 0$, $\ell \in {\mathbb N}$ and $1 \leq i \leq n$. To derive uniform estimates we argue the induction of $\ell$, taking $T$ small.

%%%%%%%%%%%%%%%%%%%%%%%%%%%%%%%%%%%%%%%%%%%%%%%%%%%%%%%%%%
%
% \ell = 1
%
%%%%%%%%%%%%%%%%%%%%%%%%%%%%%%%%%%%%%%%%%%%%%%%%%%%%%%%%%%

\vspace{2pt} \noindent {\underline{$\ell = 1$}} For $0 \leq u_0(x), v_0(x), w_0(x) \leq M$, by the maximum principle and the fact that $e^{t(d \Delta -m)} = e^{-mt} e^{dt \Delta}$, we easily see that
\[
  0 \leq u_1 (x, t) \leq \| u_0 \|, \quad 0 \leq v_1 (x, t) \leq \| v_0 \|, \quad 0 \leq w_1 (x, t) \leq \| w_0 \|
 \]
for ${\mathbb R}^n$ and $t > 0$ by $m, \rho > 0$. In addition, it is also easy to obtain that
\[
  t^{1/2} \| \partial_i u_1 (t) \| \leq \| u_0 \|, \quad (dt)^{1/2} \| \partial_i v_1 (t) \| \leq \| v_0 \|, \quad \| \partial_i w_1 (t) \| \leq \| \partial_i w_0 \|
\]
for $t > 0$ and $1 \leq i \leq n$ by the estimate of the heat kernel. Thus,
\begin{equation}\label{k1}
  K_{j, 1} \leq M \quad \text{for} \quad T > 0, \quad 1 \leq j \leq 6 \quad \text{and} \quad 1 \leq i \leq n.
\end{equation}

%%%%%%%%%%%%%%%%%%%%%%%%%%%%%%%%%%%%%%%%%%%%%%%%%%%%%%%%%%
%
% \ell = 2
%
%%%%%%%%%%%%%%%%%%%%%%%%%%%%%%%%%%%%%%%%%%%%%%%%%%%%%%%%%%

\vspace{2pt} \noindent {\underline{$\ell = 2$}} Before estimating $u_2$ and $v_2$, we shall confirm bounds for time-evolution operators $U_1$ and $V_1$. By $u_1 \geq 0$ and (\ref{k1}), it holds that
\[
  \| \eta_1 (t) \| \leq M + \frac{\gamma M}{h} =: \overline \eta_1 \quad {\text{with}} \quad \eta_1 (x, t) := u_1 (x, t) + \frac{\gamma v_1 (x, t)}{u_1 (x, t) + h}
\]
for $t > 0$. By Lemma~\ref{lem}, for $\big\{ U_1 (t, s) \big\}_{t \geq s \geq 0}$ with $A_1 (x, t) := \Delta - \eta_1 (x, t)$, we have $\displaystyle 0 \leq U_1 (t, s) u_0 \leq \frac{4}{3} \| u_0 \|$ for $x \in {\mathbb R}^n$ and $0 \leq s \leq t \leq T_2'$ with some $T_2' > 0$ depending only on $\overline \eta_1$. So, by (\ref{iu2}) with $\ell = 1$, we have
\[
  0 \leq u_2 (t) \leq \| U_1 (t, 0) u_0 \| + \int_0^t \| U_1 (t, s) \zeta_1(s) \| ds \leq 2 M
\]
with $\zeta_1 (x, t) := u_1 (x, t)$ and $0 \leq \zeta_1 (x, s) \leq \overline \zeta_1 := M$, provided if $0 \leq s \leq t \leq T_2^\dagger$ with $T_2^\dagger := \min \big\{ T_2', 1/2 \big\}$. Similarly, since
\[
  \| \xi_1 (t) \| \leq m + M =: \overline \xi_1 \quad {\text{with}} \quad \xi_1 (x, t) := m + v_1 (x, t)
\]
for $t > 0$, let $\big\{ V_1 (t, s) \big\}_{t \geq s \geq 0}$ be the time-evolution operator associated with $B_1 (x, t) := d \Delta - \xi_1 (x, t)$, we see that $\displaystyle 0 \leq V_1 (t, s) v_0 \leq \frac{4}{3} \| v_0 \|$ for $0 \leq s \leq t \leq T_2^\sharp$ with some $T_2^\sharp > 0$ depending only on $\overline \xi_1$. So, by (\ref{iv2}),
\[
  0 \leq v_2 (t) \leq \| V_1 (t, 0) v_0 \| + \int_0^t \| V_1 (t, s) \chi_1(s) \| ds \leq 2 M
\]
hold with $\chi_1 (x, t) := \mu u_1 (x, t) v_1 (x, t) / \{ u_1 (x, t) + h \} + \alpha w_1 (x, t)$ and $0 \leq \chi_1 (x, s) \leq \overline \chi_1 := (\mu M/h + \alpha) M$, provided if $0 \leq s \leq t \leq T_2^\flat$ with $T_2^\flat := \min \big\{ T_2^\dagger, T_2^\sharp, h/(2 \mu M + 2 \alpha h) \big\}$. For the estimate of $w_2$, we obtain
\[
  0 \leq w_2 (t) \leq \| e^{- \rho t} w_0 \| + \int_0^t e^{- \rho (t-s)} \| \nu u_1 v_1 / (u_1 + h) + \theta v_1 \| ds \leq 2 M
\]
for $0 \leq s \leq t \leq T_2^\natural$ with $T_2^\natural := \min \big\{ T_2^\flat, h/(\nu M + h \theta) \big\}$. To derive the estimate for $\partial_i u_2$, we use the heat semigroup expression:
\[
  u_2 (t) = e^{t \Delta} u_0 + \int_0^t e^{(t-s) \Delta} \left[ \zeta_1 - \eta_1 u_2 \right] (s) \, ds,
\]
rewriting (\ref{iu2}). Hence, it holds that
\[
  t^{1/2} \| \partial_i u_2 (t) \| \leq \| u_0 \| + t^{1/2} \int_0^t (t-s)^{-1/2} \left[ \overline \zeta_1 + \overline \eta_1 \| u_2 \| \right] ds \leq 2M
\]
for $t \in (0, T_2^\heartsuit]$ with $T_2^\heartsuit := \min \{ T_2^\natural, h / (2 h + 4 h M + 4 \gamma M) \}$. As similar way, for $\partial_i v_2$, we appeal to the heat semigroup expression again:
\begin{align*}
  & (dt)^{1/2} \| \partial_i v_2 (t) \| \\
  & \qquad \leq (dt)^{1/2} \| \partial_i e^{d t \Delta} v_0 \| + (dt)^{1/2} \int_0^t \| \partial_i e^{d (t-s) \Delta} \left[ \chi_1 - \xi_1 v_2 \right] \| ds \\
  & \qquad \leq \| v_0 \| + t^{1/2} \int_0^t (t-s)^{-1/2} \left[ \overline \chi_1 + \overline \xi_1 2 M \right] ds \leq 2M
\end{align*}
for $t \in (0, T_2^\diamondsuit]$ with $T_2^\diamondsuit := \min \{ T_2^\heartsuit, h / (2 \mu M + 2 \alpha h + 4 hm + 4 h M) \}$. Furthermore, we see that
\begin{align*}
  \partial_i w_2 (t) & = e^{- \rho t} \partial_i w_0 \\
  & \quad + \int_0^t e^{- \rho (t-s)} \left[ \frac{\nu h (\partial_i u_1) v_1 + \nu u_1 (\partial_i v_1) (u_1 + h)}{(u_1 + h)^2} + \theta \partial_i v_1 \right] ds
\end{align*}
holds true, and this implies that
\begin{align*}
  \| \partial_i w_2 (t) \| & \leq M + \int_0^t \left\{ \frac{\nu h \sqrt{d} M + \nu M (M + h)}{h^2} + \theta \right\} M (d s)^{-1/2} ds \\
  & \leq 2 M \quad {\text{for}} \quad t \in [0, T_2]
\end{align*}
with $T_2 := \min \{ T_2^\diamondsuit, d h^4/[4 \nu h \sqrt{d} M + 4 \nu M^2 + 4 \nu h M + 4 h^2 \theta]^2 \}$.

Therefore, it is shown that $u_2, v_2, w_2 \geq 0$ and
\begin{equation}\label{k2}
  K_{j, 2} \leq 2 M \quad \text{for} \quad t \in (0, T_2], \quad 1 \leq j \leq 6 \quad \text{and} \quad 1 \leq i \leq n.
\end{equation}

%%%%%%%%%%%%%%%%%%%%%%%%%%%%%%%%%%%%%%%%%%%%%%%%%%%%%%%%%%
%
% \ell = 3
%
%%%%%%%%%%%%%%%%%%%%%%%%%%%%%%%%%%%%%%%%%%%%%%%%%%%%%%%%%%

\vspace{2pt} \noindent {\underline{$\ell = 3$}} We stand for the time-evolution operator $\big\{ U_2 (t, s) \big\}_{t \geq s \geq 0}$ with $A_2 (x, t) := \Delta - \eta_2 (x, t)$ and $\eta_2 (x, t) := u_2 (x, t) + \gamma v_2 (x, t)/ \{ u_2 (x, t) + h \}$. By Lemma~\ref{lem}, $U_2 (t, s) u_0 \geq 0$ and $\| U_2 (t, s) \|_{L^\infty \to L^\infty} \leq 4/3$ for $0 \leq s \leq t \leq T_3'$ with some $T_3' > 0$, since $0 \leq \eta_2 (x, t) \leq \overline \eta := 2 M + 2 \gamma M/h$ by (\ref{k2}). So, we get
\[
  0 \leq u_3 (x, t) \leq \| U_2 (t, 0) u_0 \| + \int_0^t \| U_2 (t, s) \zeta_2 (s) \| ds \leq 2 M
\]
for $x \in {\mathbb R}^n$ and $t \in [0, T_3^\dagger]$ with $T_3^\dagger := \min \{ T_3', 1/4 \big\}$. Here we have used that $0 \leq \zeta_2 (x, t) := u_2 (x, t) \leq \overline \zeta := 2 M$. Similarly, denote the time-evolution operator by $\big\{ V_2 (t, s) \big\}_{t \geq s \geq 0}$ associated with $B_2 (x, t) := d \Delta - \xi_2 (x, t)$. Since $0 \leq \xi_2 (x, t) := m + v_2 (x, t) \leq \overline \xi := m + 2 M$, $V(t, s) v_0 \geq 0$ and $\| V_2 (t, s) \|_{L^\infty \to L^\infty} \leq 4/3$ holds for $0 \leq s \leq t \leq T_3^\sharp$ with some $T_3^\sharp > 0$ by Lemma~\ref{lem}. Hence, we can see that
\[
  0 \leq v_3 (x, t) \leq \| V_2 (t, 0) v_0 \| + \int_0^t \| V_2 (t, s) \chi_2 (s) \| ds \leq 2 M
\]
for $x \in {\mathbb R}^n$ and $t \in [0, T_3^\flat]$ with $T_3^\flat := \min \big\{ T_3^\dagger, T_3^\sharp, h/(8 \mu M + 4 \alpha h) \big\}$. Here we have used $0 \leq \chi_2 (x, t) := \mu u_2 (x, t) v_2 (x, t) / \{ u_2 (x, t) + h \} + \alpha w_2 (x, t) \leq \overline \chi := 4 \mu M^2/h + 2 \alpha M$ by (\ref{k2}). It is also easy to show that
\[
  0 \leq w_3 (x, t) \leq \| w_0 \| + \int_0^t \| \nu u_2 v_2 / (u_2 + h) + \theta v_2 \| ds \leq 2 M
\]
for $x \in {\mathbb R}^n$ and $t \in [0, T_3^\natural := \min \big\{ T_3^\flat, h/(4 \nu M + 2 h \theta) \big\}$. By the heat semigroup expression, we obtain that
\[
  t^{1/2} \| \partial_i u_3 (t) \| \leq \| u_0 \| + t^{1/2} \int_0^t (t-s)^{-1/2} \left[ \| \zeta_2 \| + \| \eta_2 u_3 \| \right] ds \leq 2M
\]
for $t \in (0, T_3^\heartsuit]$ with $T_3^\heartsuit := \min \{ T_3^\natural, h / (4 h + 8 h M + 8 \gamma M) \}$. As similar way, we derive
\[
  (dt)^{1/2} \| \partial_i v_3 (t) \| \leq \| v_0 \| + t^{1/2} \int_0^t (t-s)^{-1/2} \left[ \| \chi_2 \| + \| \xi_2 v_3 \| \right] ds \leq 2M
\]
for $t \in (0, T_3^\diamondsuit]$ with $T_3^\diamondsuit := \min \{ T_3^\heartsuit, h / (4 h m + 8 h M + 8 \mu M + 4 \alpha h) \}$. For estimate $\partial_i w_3$, we have
\begin{align*}
  \| \partial_i w_3 (t) \| & \leq M + \int_0^t \left\| \frac{\nu h (\partial_i u_2) v_2 + \nu u_2 (\partial_i v_2) (u_2 + h)}{h^2} + \theta \partial_i v_2 \right\| ds \\
  & \leq 2 M \quad {\text{for}} \quad t \in (0, T_0]
\end{align*}
with $T_0 := \min \{ T_3^\diamondsuit, d h^4/[8 \nu h \sqrt{d} M + 16 \nu M^2 + 8 \nu h M + 4 h^2 \theta]^2 \}$. Note that the estimate $T_0 \geq C/(M^4+1)$ is yielded with some $C > 0$.

Therefore, we see that $u_3, v_3, w_3 \geq 0$ and
\[
  K_{j, 3} \leq 2M \quad \text{for} \quad t \in (0, T_0], \quad 1 \leq i \leq n \quad {\text{and}} \quad 1 \leq j \leq 6.
\]

%%%%%%%%%%%%%%%%%%%%%%%%%%%%%%%%%%%%%%%%%%%%%%%%%%%%%%%%%%
%
% \ell = 4, 5, ...
%
%%%%%%%%%%%%%%%%%%%%%%%%%%%%%%%%%%%%%%%%%%%%%%%%%%%%%%%%%%

\vspace{2pt} \noindent {\underline{$\ell = 4, 5, \ldots$}} Let $\ell \geq 4$. We assume that $u_\ell$, $v_\ell$, $w_\ell \geq 0$ and
\begin{equation}\label{k-ell}
  K_{j, \ell} \leq 2 M \quad \text{for} \quad t \in (0, T_0], \quad 1 \leq j \leq 6 \quad {\text{and}} \quad 1 \leq i \leq n
\end{equation}
hold true. We shall compute estimates for $u_{\ell + 1}$, $v_{\ell + 1}$ and $w_{\ell + 1}$. Note that $\eta_\ell \leq \overline \eta$, $\zeta_\ell \leq \overline \zeta$, $\xi_\ell \leq \overline \xi$, $\chi_\ell \leq \overline \chi$ hold, independently of $\ell \geq 3$. So, as the same discussion in the case $\ell = 3$ in above, one can see that $u_{\ell+1}$, $v_{\ell+1}$, $w_{\ell+1} \geq 0$ and
\[
  K_{j, \ell+1} \leq 2 M \quad \text{for} \quad t \in (0, T_0], \quad 1 \leq j \leq 6 \quad {\text{and}} \quad 1 \leq i \leq n.
\]
The detail is omitted in here. Hence, the non-negativities of approximation and $(\ref{k-ell})$ hold true for all $\ell \in {\mathbb N}$.

One may see that $\big( u_\ell, v_\ell, w_\ell \big)$ are continuous in $t \in [0, T_0]$ for $\ell \in {\mathbb N}$. It is also easy to see that $\big\{ u_\ell, v_\ell, w_\ell, t^{1/2} \partial_i u_\ell, t^{1/2} \partial_i v_\ell, \partial_i w_\ell \big\}_{\ell = 1}^\infty$ are Cauchy sequences in $C([0, T_0]; BUC)$, choosing $T_0$ small again, if necessary. Let
\[
  \big( u, v, w, \hat u, \hat v, \hat w \big) := \lim_{\ell \to \infty} \big( u_\ell, v_\ell, w_\ell, t^{1/2} \partial u_\ell, t^{1/2} \partial v_\ell, \partial_i w_\ell \big).
\]
Obviously, the coincidences $\hat u = t^{1/2} \partial_i u$, $\hat v = t^{1/2} \partial_i v$ and $\hat w = \partial_i w$ hold by construction. Furthermore, it is also ensured that
\[
  0 \leq u (x, t), v (x, t), w (x, t) \leq 2M \quad {\text{for}} \,\,\, x \in {\mathbb R}^n \,\,\, {\text{and}} \,\,\, t \in [0, T_0].
\]

The uniqueness follows from $(\ref{int-u})-(\ref{int-w})$ and Gronwall's inequality, directly. If fact, if $\big( u, v, w \big)$ and $\big( u^\ast, v^\ast, w^\ast \big)$ are solutions to (P) with the same initial data $\big( u_0, v_0, w_0 \big)$, then $u \equiv u^\ast$, $v \equiv v^\ast$ and $w \equiv w^\ast$ simultaneously hold. Thanks to the boundedness of the first derivatives, it is easy to control the second derivatives in $x$ of $u$ and $v$ for $t \in (0, T_0]$, as well as the first derivatives in $t$ of solutions. So, we may see that $\big( u, v, w \big)$ is a triplet of time-local unique classical solutions to (P). This completes the proof of Proposition~\ref{tlwp}.
\end{proof}

%%%%%%%%%%%%%%%%%%%%%%%%%%%%%%%%%%%%%%%%%%%%%%%%%%%%%%%%%%
%
% Remark 2
%
%%%%%%%%%%%%%%%%%%%%%%%%%%%%%%%%%%%%%%%%%%%%%%%%%%%%%%%%%%

\begin{remark}\label{r2}{\rm
(i)~If $w_0$ is smooth, then $\big( u, v, w \big)$ is smooth in $x$ and $t$.

\noindent (ii)~For $d = 0$, we can also get time-local well-posedness, if $v_0 \in BUC^1$.

\noindent (iii)~The instability of the trivial solution is easily obtained. Moreover, by strong maximum principle for solutions to the heat equation, $u > 0$ for $x \in {\mathbb R}^n$ and $t \in (0, T_0]$, if $u_0 \not\equiv 0$. This means that ${\rm{supp}} \, u(t) = {\mathbb R}^n$ for any small $t > 0$, even if ${\rm{supp}} \, u_0$ is compact. That is, the propagation speed of solutions to (P) is infinite, as the same as the heat equation. In addition, $v > 0$ and $w > 0$ for $t > 0$, if either $v_0 \not\equiv 0$ or $w_0 \not\equiv 0$.
}\end{remark}

%%%%%%%%%%%%%%%%%%%%%%%%%%%%%%%%%%%%%%%%%%%%%%%%%%%%%%%%%%
%
% Section 4. time-global well-posedness
%
%%%%%%%%%%%%%%%%%%%%%%%%%%%%%%%%%%%%%%%%%%%%%%%%%%%%%%%%%%

\section{time-global well-posedness}

\noindent In this section, we will derive a priori bounds of solutions and their derivatives. To do so, our first task is to obtain upper bounds of solutions to (P) with large initial data. For the case when $\| u_0 \| \leq 1$, we will discuss in Remark~\ref{r3}-(ii) in below and Section~5.

%%%%%%%%%%%%%%%%%%%%%%%%%%%%%%%%%%%%%%%%%%%%%%%%%%%%%%%%%%
%
% Proposition 2. upper bounds
%
%%%%%%%%%%%%%%%%%%%%%%%%%%%%%%%%%%%%%%%%%%%%%%%%%%%%%%%%%%

\begin{proposition}\label{ub}
Suppose the assumption of Proposition~$\ref{tlwp}$. If $\kappa_0 := \| u_0 \| > 1$, then $0 < u (x, t) < \kappa_0$, $0 \leq v \leq \max \big\{ \| v_0 \|, \widetilde v \big\}$ and $0 \leq w \leq \max \big\{ \| w_0 \|, \widetilde w \big\}$ for $x \in {\mathbb R}^n$ and $t > 0$, as long as the classical solutions exist. Here, $\widetilde v := \mu \kappa_0 / (\kappa_0 + h) + \alpha (\nu \kappa_0 + \theta \kappa_0 + \theta h) / (\rho \kappa_0 + \rho h) - m$ and $\widetilde w := (\nu \kappa_0 + \theta \kappa_0 + \theta h) \widetilde v / (\rho \kappa_0 + \rho h)$.
\end{proposition}

\begin{proof}
If $v_0 \equiv 0$ and $w_0 \equiv 0$, then $v \equiv w \equiv 0$ for $t > 0$. Assume either $v_0 \not\equiv 0$ or $w_0 \not\equiv 0$. So, as seen in Remark~\ref{r2}-(iii), we have $u$, $v$, $w > 0$. The behavior of $u$ will be observed. Consider the logistic equation:
\begin{equation}\label{kappa}
  \kappa' = (1 - \kappa) \kappa, \quad \kappa (0) = \kappa_0 > 1.
\end{equation}
By maximum principle, $u (x, t) \leq \kappa (t)$ holds for $x \in {\mathbb R}^n$ and $t > 0$, as long as the classical $u$ exists. Since $\kappa (t) = \kappa_0 / (\kappa_0 + e^{-t} - \kappa_0 e^{-t}) < \kappa_0$ for $t > 0$, it is clear that $u < \kappa_0$.

Next, we investigate on $v$. One may assume $\widetilde v > 0$, without loss of generality. Furthermore, we assume $\omega_0 := \| v_0 \| > \widetilde v$. Let $\omega = \omega (t)$ be a solution to
\begin{equation}\label{omega}
  \omega' = \mu \frac{\kappa_0 \omega}{\kappa_0 + h} + \frac{\alpha}{\rho} \left( \nu \frac{\kappa_0 \omega}{\kappa_0 + h} + \theta \omega \right) - (m + \omega) \omega = (\widetilde v - \omega) \omega
\end{equation}
with $\omega (0) = \omega_0 > \widetilde v$. Since $\omega$ is monotone decreasing, $\widetilde v < \omega (t)  < \omega_0$ for $t > 0$. By maximum principle, $v (x, t) \leq \omega (t)$ holds for $x \in {\mathbb R}^n$ and $t > 0$. Here, we have used $0 < u/(u+h) \leq \kappa_0/(\kappa_0 + h)$ by $0 < u \leq \kappa_0$. This yields that $v < \omega_0$. If $\| v_0 \| \leq \widetilde v$, then it is easy to see that $v \leq \widetilde v$.

Since $u \leq \kappa_0$ and $v \leq \widetilde v$, one can easily see that $w \leq \widetilde w$, when $\| w_0 \| \leq \widetilde w$. Also, even if $\| w_ 0 \| > \widetilde w$, then $w < \| w_0 \|$ holds by the same observation above.
\end{proof}

%%%%%%%%%%%%%%%%%%%%%%%%%%%%%%%%%%%%%%%%%%%%%%%%%%%%%%%%%%
%
% Remark 3
%
%%%%%%%%%%%%%%%%%%%%%%%%%%%%%%%%%%%%%%%%%%%%%%%%%%%%%%%%%%

\begin{remark}\label{r3}{\rm
(i)~By definition, it is clear that $\widetilde v \geq \overline v$ and $\widetilde w \geq \overline w$, if $\kappa_0 \geq 1$. Besides, $\widetilde v \leq \overline v$ and $\widetilde w \leq \overline w$, if $\kappa_0 \leq 1$.

\noindent (ii)~Even $\kappa_0 =  \| u_0 \| \leq 1$, the uniform bounds on $v$ and $w$ are obtained; $v \leq \| v_0 \|$ holds with $\| v_0 \| \geq \overline v$, and $w \leq \| w_0 \|$ holds with $\| w_0 \| \geq \overline w$.
}\end{remark}

In what follows, we will give the a priori estimate for $\| \partial_i w (t) \|$, which may grow in $t$. As seen in Proposition~\ref{ub}, we prove that $0 \leq u, v, w \leq N$ as long as the classical solutions exist, if $N$ is chosen as
\[
  N := \max \big\{ 1, \kappa_0, \overline v, \widetilde v, \| v_0 \|,\overline w, \widetilde w, \| w_0 \| \big\} \quad {\text{with}} \quad \kappa_0 := \| u_0 \|.
\]

%%%%%%%%%%%%%%%%%%%%%%%%%%%%%%%%%%%%%%%%%%%%%%%%%%%%%%%%%%
%
% Proposition 3. estimates of derivatives
%
%%%%%%%%%%%%%%%%%%%%%%%%%%%%%%%%%%%%%%%%%%%%%%%%%%%%%%%%%%

\begin{proposition}\label{d_iw}
If $0 \leq u, v, w \leq N$ for $x \in {\mathbb R}^n$ and $t \in [0, T]$ with $N$ and $T$, then there exists a $C > 0$ $($independent of $N$ and $T)$ such that
\[
  \| \partial_i w (t) \| \leq \| \partial_i w_ 0 \| + C (N^4 + N) \left( t^{1/2} + t^{3/2} \right), \quad t \in [0, T], \,\,\, 1 \leq i \leq n.
\]
\end{proposition}

\begin{proof}
We first derive the estimate for $\partial_i u$. By (\ref{int-u}), we have
\begin{align*}
  \| \partial_i u (t) \| & \leq \| u_0 \| t^{-1/2} + \int_0^t (t-s)^{-1/2} \left\| (1-u) u - \frac{\gamma u v}{u+h} \right\| ds \\
  & \leq C (N^2 + N) (t^{-1/2} + t^{1/2})
\end{align*}
for $t \in [0, T]$ and $1 \leq i \leq n$ with some $C$. Similarly, by (\ref{int-v}), we seek
\begin{align*}
  & \| \partial_i v (t) \| \\
  & \quad \leq \| v_0 \| (d t)^{-1/2} + \int_0^t (dt-ds)^{-1/2} \left\| \frac{\mu u v}{u + h} + \alpha w - (m + v) v \right\| ds \\
  & \quad \leq C (N^2 + N) (t^{-1/2} + t^{1/2})
\end{align*}
with some $C$. Finally, by (\ref{int-w}) and estimates above, it turns out that
\begin{align*}
  \| \partial_i w (t) \| & \leq \| \partial_i w_0 \| + \int_0^t \left\| \frac{\nu h (\partial_i u) v + \nu u (\partial_i v) (u + h)}{(u + h)^2} + \theta \partial_i v \right\| ds \\
  & \leq \| \partial_i w_0 \| + C (N^4 + N) \int_0^t (s^{-1/2} + s^{1/2}) ds \\
  & \leq \| \partial_i w_0 \| + C (N^4 + N) (t^{1/2} + t^{3/2})
\end{align*}
for $t \in [0, T]$ and $1 \leq i \leq n$ with some positive constant $C$ depending on parameters, however, independent of $N$ and $T$.
\end{proof}

Note here that Theorem~\ref{th} follows from Proposition~\ref{ub}, Proposition~\ref{d_iw} and the estimate $T_0 \geq C/(M^4+1)$ in Proposition~\ref{tlwp}, since we can extend the obtained unique solutions time-globally, repeating construction.

%%%%%%%%%%%%%%%%%%%%%%%%%%%%%%%%%%%%%%%%%%%%%%%%%%%%%%%%%%
%
% Section 5. invariant regions
%
%%%%%%%%%%%%%%%%%%%%%%%%%%%%%%%%%%%%%%%%%%%%%%%%%%%%%%%%%%

\section{Invariant regions}

\noindent This section will be devoted to observe invariant regions. The proof of Theorem~\ref{th2}-(i) is easy, since $\big( 1, 0, 0 \big)$ is only one stable constant state. So, we skip it in here.

We are now position to give a proof of Theorem~\ref{th2}-(ii). The key step is to deduce a priori bounds of solutions, due to the maximum principle and comparison with solutions to the system of corresponding ordinary differential equations (\ref{kappa}) and (\ref{omega}). Let us recall the assumption:
\[
  \overline v := \mu/(1+h) + \alpha (\nu + \theta + \theta h) / (\rho + \rho h) - m > 0,
\]
$\overline w := (\nu + \theta + \theta h) \overline v / (\rho + \rho h) > 0$ and $R := [ 0, 1 ] \times [ 0 , \overline v ] \times [ 0, \overline w ]$.

%%%%%%%%%%%%%%%%%%%%%%%%%%%%%%%%%%%%%%%%%%%%%%%%%%%%%%%%%%
%
% proof of Theorem 2-(ii)
%
%%%%%%%%%%%%%%%%%%%%%%%%%%%%%%%%%%%%%%%%%%%%%%%%%%%%%%%%%%

\begin{proof}[Proof of Theorem~{\rm{\ref{th2}-(ii)}}]
We first show that $R$ is an invariant region. Let $\big( u_0, v_0, w_0 \big) \in R$. By construction of time-local solutions in Proposition~\ref{tlwp}, the non-negativity of solutions is clarified. Note that $\big( 0, 0, 0 \big)$ and $\big( 1, 0, 0 \big)$ are classical solutions in $R$. If $u_0 \equiv 0$, then $u \equiv 0$, in addition, $v \in [0, \overline v]$ and $w \in [0, \overline w]$, since $v^\flat := \alpha \theta / \rho - m \leq \overline v$ and $w^\flat := \theta (\alpha \theta - m \rho) / \rho^2 \leq \overline w$. Also, it is easy to see that $v \equiv 0$ and $w \equiv 0$ hold for $t > 0$, provided if $v_0 \equiv 0$ and $w_0 \equiv 0$.

Let $u_0 \not\equiv 0$ and either $v_0 \not\equiv 0$ or $w_0 \not\equiv 0$. As seen in Remark~\ref{r2}-(iii), it is clear that the classical solutions $u$, $v$, $w$ never touch to $0$, as long as they exist. Moreover, with $u_0 \leq 1$, we observe that $u(\tau) < 1$ for small $\tau > 0$ by the strong maximum principle. Similarly, it turns out that $v(\tau) < \overline v$ by $v_0 \leq \overline v$, as well as $w(\tau) < \overline w$. So, regarding $\tau$ as the initial time, one may assume $\big( u_0, v_0, w_0 \big) \in R^\circ := R \setminus \partial R$, without loss of generality.

Put $(\hat x, \hat t) \in {\mathbb R}^n \times (0, T_0]$ such that $\hat t$ is the first time when $u$ touches to $1$ at $\hat x$. We may assume $|\hat x| < \infty$ by Oleinik's argument; see e.g. \cite{GG99}. Since $u (\hat x, \hat t) = 1$ is the local maximum, at $(\hat x, \hat t)$ we see that $\partial_t u \geq 0$, $\Delta u \leq 0$, $(1 - u) u = 0$ and $-\gamma uv/(u+h) < 0$ by $v > 0$. This contradicts to the fact that $u$ is a solution to (P). Hence, $u$ never touches to $1$.

The same argument works on $v$. Indeed, let $0 < u < 1$, $0 < w < \overline w$, and let $(\check x, \check t) \in {\mathbb R}^n \times (0, T_0]$ such that $\check t$ is the first time when $v$ touches to $\overline v$ at $\check x$. So, at $(\check x, \check t)$, we see that $\partial_t v \geq 0$, $d \Delta v \leq 0$ and
\[
  \frac{\mu u v}{u+h} + \alpha w - (m + v) v < \frac{\mu \overline v}{1+h}  + \alpha \overline w - (m + \overline v) \overline v = 0.
\]
So, $v$ never touches to $\overline v$. As the same as above, one may confirm that $w$ never touches to $\overline w$ as long as classical solutions exist. This means that a triplet of the solutions always remains in $R^\circ \subset R$.

Next, we show the asymptotic behavior of solutions, briefly. Even if $\| u_0 \| > 1$, by $u (x, t) \leq \kappa (t)$, then there exists a $T_\varepsilon^\ast > 0$ such that $\| u (t) \| < 1 + \varepsilon$ for $t > T_\varepsilon^\ast$. From this and (\ref{omega}), there exists $T_\varepsilon^\sharp > T_\varepsilon^\ast$ such that $\| v (t) \| < \overline v + \varepsilon$ for $t > T_\varepsilon^\sharp$. Finally, one can also show that there exists $T_\varepsilon > T_\varepsilon^\sharp$ such that $\| w (t) \| < \overline w + \varepsilon$ for $t > T_\varepsilon$, by similar way. This completes the proof of Theorem~\ref{th2}-(ii).
\end{proof}

The proof of Theorem~\ref{th2}-(iii) is essentially similar to above. So, we omit it in here.

%%%%%%%%%%%%%%%%%%%%%%%%%%%%%%%%%%%%%%%%%%%%%%%%%%%%%%%%%%
%
% Remark 4
%
%%%%%%%%%%%%%%%%%%%%%%%%%%%%%%%%%%%%%%%%%%%%%%%%%%%%%%%%%%

\begin{remark}\label{r4}{\rm
The stability of non-trivial constant states to the system of corresponding ordinary differential equations can be obtained, using linear algebra. For example, if we choose
\[
  \mu = \nu = \frac{\gamma}{2}, \, m = \theta = 0, \, \alpha = \rho = \frac{1}{4}, \, \gamma = h + \frac{1}{2}, \leqno{\rm{(E1)}}
\]
then a constant state $\big( u, v, w \big) = \big( 1/2, 1/2, 1/2 \big)$ is linearly stable for any $h > 0$. On the other hand, if we select the parameters as
\[
  \mu = \frac{3 \gamma}{4}, \, \nu = \frac{\gamma}{2}, \, m = \frac{\theta}{2} = \frac{1}{8}, \, \alpha = \frac{\rho}{2} = \frac{1}{4}, \, \gamma = h + \frac{1}{2}, \leqno{\rm{(E2)}}
\]
then $\big( 1/2, 1/2, 1/2 \big)$ is again a constant state whose stability is bifurcated in $h$. Indeed, $\big( 1/2, 1/2, 1/2 \big)$ is linearly instable for $h \in (0, 1/2)$, however, linearly stable for $h > 1/2$. The authors believe that such stability is still valid for solutions to (P). For studying Turing instability, it is needed to consider more complicated situation, e.g. when $\mu$ and $\nu$ are functions of $u$.
}\end{remark}

%%%%%%%%%%%%%%%%%%%%%%%%%%%%%%%%%%%%%%%%%%%%%%%%%%%%%%%%%%
%
% Bibliography
%
%%%%%%%%%%%%%%%%%%%%%%%%%%%%%%%%%%%%%%%%%%%%%%%%%%%%%%%%%%

\end{document}